\documentclass[letterpaper,10pt]{article}

\usepackage[affil-it]{authblk}
\usepackage[margin=1.5in]{geometry}

\usepackage[toc,page]{appendix}
\usepackage{amssymb,amsmath,latexsym,amsthm,amsfonts}
\usepackage{cite} 

\usepackage[pdftex,pagebackref=true,colorlinks=true,linkcolor=blue,unicode]{hyperref} 

\usepackage[normalem]{ulem} 

\usepackage{enumerate}

\newcommand{\I}[1]{ {\bf 1}_{\left\{ #1 \right\}} }
\newcommand{\R}{\mathbb{R}}
\newcommand{\N}{\mathbb{N}}
\newcommand{\EE}{\mathbb{E}}
\newcommand{\PP}{ \mathbb{P}}
\newcommand{\QQ}{\mathbb{Q}}

\newcommand{\dd}{\mathrm{d}}

\newcommand{\qqq}{q}
\newcommand{\ppp}{p}
\newcommand{\eo}{\epsilon_1}
\newcommand{\et}{\epsilon_2}

\newtheorem{theorem}{Theorem}
\newtheorem{lemma}{Lemma}
\newtheorem{proposition}{Proposition}
\newtheorem{corollary}{Corollary}

\newtheorem{definition}{Definition}

\newcommand{\guidonote}[1]{#1}

\newcommand{\bast}{{b,\infty}}
\newcommand{\bc}{{b,c}}

\title{A Dichotomy for Sampling Barrier-Crossing Events of Random Walks with Regularly Varying Tails}
\author{A.~B.~Dieker\thanks{Postal address: Department of Industrial Engineering and Operations Research, Columbia University, New York, NY, 10027, USA}}
\affil{Columbia University}
\author{Guido R.~Lagos\thanks{Postal address: Faculty of Engineering and Sciences, Universidad Adolfo Ib\'a\~nez, Pe\~nalol\'en, Santiago, Chile}}
\affil{Universidad Adolfo Ib\'a\~nez}
\date{\today}

\begin{document}

\maketitle

\begin{abstract}
\guidonote{We study how to sample} paths of a random walk up to the first time it crosses a fixed barrier, in the setting where the step sizes are iid with negative mean and have a regularly varying right tail. 
\guidonote{We introduce a desirable property for a change of measure to be suitable for exact simulation.
We study whether the change of measure of Blanchet~and Glynn~\cite{blanchet2008efficient} satisfies this property and show that it does so}
if and only if the tail index $\alpha$ of the right tail lies in the interval $(1, \, 3/2)$.
\end{abstract}

\section{Introduction}\label{sec:introduction}

Barrier-crossing events of random walks appear in numerous engineering and science models.
Examples range from stationary waiting times in queues to ruin events in insurance risk processes~\cite{asmussen1996large,resnick1997heavy,embrechts1997modelling}.
Random walks with regularly varying step size distributions are of particular interest, and their special analytic structure facilitates an increasingly complete understanding of associated rare events.

This paper considers the problem of sampling a path of a random walk until it crosses a given fixed barrier in the setting of heavy-tailed step sizes with negative mean.
The higher the barrier, the lower the likelihood of reaching it.
This poses challenges for conditional sampling, since naive Monte Carlo sampling devotes much computational time to paths that never cross the barrier and must therefore be ultimately discarded.

The ability to sample up to the first barrier-crossing time plays a central role in several related problems, such as for sampling paths up to their maximum~\cite{blanchet2011exact} or for sampling only the maximum itself~\cite{ensor2000simulating}.
In turn these have applications to perfect sampling from stationary distributions~\cite{blanchet2012steady,blanchet2014exact} and to approximately solving stochastic differential equations~\cite{dong2014strong}.



\paragraph{Main contributions.}
The central question in this paper is: Can the change of measure proposed in Blanchet and Glynn~\cite{blanchet2008efficient} be used for \guidonote{\emph{exact}, i.e.~unbiased,} conditional sampling of heavy-tailed random walks given a barrier-crossing rare event?
The Blanchet-Glynn measure is designed to approximate such a conditional distribution, so
one might therefore expect an answer to this question in the affirmative.
Surprisingly, we answer this question by and large in the negative: this measure cannot be used for conditional sampling.
Our results are a consequence of a delicate second-order analysis of tail probabilities of a sum of heavy tailed random variables \guidonote{that are related to the \emph{residual life tail distribution} of the random walk increments. The asymptotic decomposition and analysis we put forth may be of interest in itself}.
 
\guidonote{We introduce a `desirable' property for a candidate change of measure to be suitable for our exact sampling problem and reveal an intriguing dichotomy on the suitability of} the Blanchet-Glynn measure: this measure \guidonote{satisfies this property} if and only if the tail index is below the threshold $3/2$.
It is worthwhile to stress two immediate consequences.
First, our result roughly implies that only the heaviest tails \guidonote{stand a chance to} be efficient in this setting\guidonote{, since the desirable property we introduce is intuitively a proxy for efficiency of the proposal measure}.
This is counterintuitive, since heavier tails typically make problems harder.
Second, the threshold is not directly connected to the existence of integer moments for the step size distribution.

The threshold $3/2$ also arises in the simulation literature involving barrier-crossing events with regularly varying step sizes~\cite{blanchet2012efficient,murthy2014state}.
The nature of the threshold we obtain here is however different from these works for three reasons.
First, these papers focus on estimating the rare event probability of exceeding a barrier; in contrast, our work focuses on sampling barrier-crossing paths. 
Second, these papers obtain that the heaviest tails are inefficient in their framework, while we obtain \guidonote{a form of inefficiency for lighter tails}.
Third, and perhaps most importantly, the threshold $3/2$ in the existing literature is a direct consequence of requiring second moment conditions of the estimator, while a direct relation with moments is absent for conditional sampling problems.

A by-product of our work is a counterexample for the statement of Proposition~4 of Blanchet and Glynn~\cite{blanchet2008efficient}.
This proposition states that, for a broad class of heavy-tailed step sizes, the expected hitting time of the barrier grows linearly in the barrier level under the Blanchet-Glynn change of measure.
We show, though, that this result does not always hold.
This proposition is not central to the framework introduced in~\cite{blanchet2008efficient}, and the issue we expose here can also be deduced from Corollary~1 in~\cite{blanchet2012efficient}, but our result reopens the question of when the measure of Blanchet and Glynn induces a linear hitting time expectation.


\paragraph{Related literature.}
The primary means for \guidonote{exact or unbiased} sampling from heavy-tailed random walks is based on the change of measure technique.
Simply put, this procedure consists in sampling from a distribution different from the desired one and determining (or computing) the output using the likelihood ratio.
The essential idea is that the changed \guidonote{or proposal} distribution should emphasize characteristics of barrier-crossing paths.

The literature of \guidonote{\emph{exact simulation} of barrier-crossing paths} is closely related to the one of \emph{estimating the probability} of exceeding the barrier.
In the heavy-tailed setting, the latter problem has already been studied for two decades.
In contrast, the \guidonote{exact} path-sampling problem has only recently received attention, mostly driven by applications of \emph{Dominated Coupling From the Past} when in presence of heavy tails; see~\cite{blanchet2014exact}.

For the probability estimating problem under heavy tails, early approaches are~\cite{asmussen1997simulation,asmussen2000rare,juneja2002simulating}.
An important contribution for the current paper is~\cite{blanchet2008efficient}, which was later followed by~\cite{blanchet2012efficient,murthy2014state}.
A recent new technique is~\cite{gudmundsson2014}, which uses \emph{Markov Chain Monte Carlo} to estimate the multiplicative inverse of the probability of crossing the barrier.

The \guidonote{problem of exact sampling of paths} with heavy tails, on the other hand, has only recently been tackled by~\cite{blanchet2014exact}.
The latter modifies the measure of~\cite{murthy2014state}, which focuses on the probability estimation problem, and builds on the scheme for \emph{exact sampling} of paths introduced in~\cite[\S 4]{blanchet2011exact}.
The approach studied in this paper is based on the Blanchet-Glynn change of measure, which is conceptually simpler than the approach proposed in~\cite{blanchet2014exact}.
The search for a simpler algorithm provided the motivation for this paper.


\paragraph{Outline.}
This paper is organized as follows.
In Section~\ref{sec:preliminaries} we discuss the general preliminaries for our conditional sampling problem: a change of measure technique and \guidonote{the criterion we propose as a desirable property for efficiency for exact conditional sampling}.
In Section~\ref{sec:main} we state our main result of efficiency for conditional sampling when using the Blanchet-Glynn change of measure~\cite{blanchet2008efficient} and with regularly varying step sizes.
In Section~\ref{sec:comparison} we compare our threshold result of Section~\ref{sec:main} with similar ones in the literature of rare event sampling.
In Section~\ref{sec:proof} we give a proof of the main result of Section~\ref{sec:comparison}.

\paragraph{Notation.}
We denote by $\{S_n\}$ the infinite length paths of the random walk.
Given a probability measure $\QQ$ over $\{S_n\}$, we denote the expectation with respect to measure $\QQ$ as $\EE^\QQ$.
We write $\EE^\QQ_y [\cdot] := \EE^\QQ [ \cdot | S_0=y ]$ and omit $y$ when $y=0$, as customary in the literature.
Given two probability measures $\mathsf{P}$ and $\mathsf{Q}$ over the same space, we denote \emph{absolute continuity} of $\mathsf{P}$ with respect to $\mathsf{Q}$ as $\mathsf{P} \ll \mathsf{Q}$, meaning that for all measurable $B$ $\mathsf{Q}(B)=0$ implies $\mathsf{P}(B)=0$.
For $x, y$ real we denote $x^+:=\max\{x, 0\}$, $x^-:=-\min\{x, 0\}$, $x \wedge y := \min\{x,y\}$ and $x \vee y := \max\{ x, y \}$.
Also, for two functions $f$ and $g$ we write $f(t) \sim g(t)$ when $\lim_{t \to \infty} f(t)/g(t) = 1$; we write $f(t) = O\left( g(t) \right)$ when $\limsup_{t \to \infty} |f(t)/g(t)| < \infty$, and $f(t) = o\left( g(t) \right)$ when $\lim_{t \to \infty} |f(t)/g(t)| =0$.

\section{Preliminaries}\label{sec:preliminaries}

This section gives the background necessary for the exposition of our main result.
In Section~2.1, we describe techniques for \guidonote{\emph{exact}, or \emph{unbiased},} conditional sampling using change of measure technique.
In Section~2.2, we give 
 \guidonote{the criterion we propose as a desirable property for efficiency} for this problem.
In Section~2.3, we briefly introduce the Blanchet-Glynn~\cite{blanchet2008efficient} change of measure.

\paragraph{General setting.}
We consider a random walk $S_n := \sum_{i=1}^n X_i$, where $X_i$ are iid, $\EE |X_i| < \infty$ and $S_0 = 0$ unless explicitly stated otherwise.
We assume that $\{S_n\}$ has \emph{negative drift}, meaning that $\EE X_i < 0$.
We also assume that $X_i$ has unbounded right support; that is, $\PP(X_i >t)>0$ for all $t\in \R$.

Given a \emph{barrier} $b \geq 0$, let $\tau_b := \inf \{ n \geq 0 : S_n >b \}$ be the first barrier-crossing time.
Since the random walk has negative drift, we have $S_n \to -\infty$ a.s.~as $n \to \infty$, and also $\PP (\tau_b = \infty) > 0$.

Our main goal is to study the \guidonote{suitability, with efficiency in mind, of using a change of measure to sample exactly paths} $(S_1, \ldots, S_{\tau_b})$ conditional on $\{ \tau_b < \infty \}$.

We remark that for the sake of clarity of exposition we will abuse notation and write that `$(S_0, \ldots, S_{\tau_b})$ follows the distribution $\PP(\, \cdot \, | \tau_b < \infty)$' to mean that for all finite $n \in \N$ the random vector $(S_0, \ldots, S_{\tau_b})$ with $\tau_b = n$ has the distribution $\PP(\, \cdot \, | \tau_b = n)$.

\subsection{\guidonote{Exact c}onditional sampling via change of measure}\label{sec:AR}

We tackle the problem of \guidonote{\emph{exact} or \emph{unbiased}} conditional sampling using the Acceptance-Rejection algorithm, which uses the change of measure technique.
Here we give a brief exposition of these two methods.

\paragraph{Change of measure technique.}
Let $\mathsf{P}(y, \ \dd z)$ be the \emph{transition kernel} of the random walk, i.e., $\mathsf{P}(y, \ \dd z) = \PP (S_{1} \in \ \dd z | S_0=y)$.
We consider a ``changed'' \guidonote{or ``proposal''} transition kernel $\mathsf{Q}(y, \ \dd z)$, which may be chosen \emph{state dependent}, meaning that $\mathsf{Q}(y_1,y_1+\cdot)$ and $\mathsf{Q}(y_2,y_2+\cdot)$ may be different measures for $y_1 \neq y_2$.
We assume that $\mathsf{P}(y, \cdot) \ll \mathsf{Q}(y, \cdot)$ for all $y$, which implies that the \emph{likelihood ratio} function $\dd \mathsf{P} / \dd \mathsf{Q} (y,\cdot)$ exists.
Letting $\QQ$ be the distribution of $\{S_n\}$ induced by the proposal kernel $\mathsf{Q}$, we slightly abuse notation and denote by $\dd \PP/\dd \QQ (S_n : 0 \leq n \leq T)$ the likelihood ratio of a finite path $(S_0 , \ldots, S_T)$.
More precisely, for $T$ finite $\dd \PP/\dd \QQ (S_n : 0 \leq n \leq T) := L_T$ where $L_T$ is the nonnegative random variable satisfying $\EE^\QQ \left[ {\bf 1}_B L_T \right] = \EE^\PP \left[ {\bf 1}_B \right]$ for all $B$ in the $\sigma$-algebra $\sigma (S_n : 0 \leq n \leq T)$.
With this, it holds that
\begin{eqnarray*}
\frac{\dd \PP}{\dd \QQ} (S_n : 0 \leq n \leq T) = \frac{\dd \mathsf{P}}{\dd \mathsf{Q}}(S_0,S_1) \cdots \frac{\dd \mathsf{P}}{\dd \mathsf{Q}}(S_{T-1},S_T),
\end{eqnarray*}
for all $T$ finite or $\QQ$-a.s.~finite stopping time.
See~\cite[\S XIII.3]{asmussen2003applied} for further details.

\paragraph{Acceptance-Rejection algorithm for \guidonote{exact} conditional sampling.}
This procedure considers the situation of a distribution that is ``difficult'' to sample from, and another distribution that is ``easy'' to sample from; the aim is to simulate from the difficult distribution.
The Acceptance-Rejection algorithm allows one to sample from the difficult distribution by repeatedly sampling from the easy, ``proposal'', distribution.
Here we show a known specialization of this technique to the problem of sampling paths from the conditional distribution $\PP \left(\, \cdot \, | \tau_b < \infty \right)$, see~\cite{blanchet2011exact}.

Let $\mathsf{P}$ be the transition kernel of the random walk, and consider a ``proposal'' kernel $\mathsf{Q}$, possibly state dependent, such that $\mathsf{P}(y, \cdot) \ll \mathsf{Q}(y, \cdot)$ for all $y$.
Assume that for some computable constant $C > 0$ we have
\begin{eqnarray*}
\frac{\dd \PP}{\dd \QQ} (S_n : 0 \leq n \leq \tau_b) \cdot \I{\tau_b < \infty} \leq C \qquad \QQ\text{-a.s.}
\end{eqnarray*}
If $U$ is uniformly distributed on $[0,1]$ under $\QQ$ and drawn independently from $\{S_n\}$, then it can be verified that
\begin{eqnarray}
\label{eq:AR 2}
\QQ \left( U \leq \frac{\I{\tau_b < \infty}}{C} \frac{\dd \PP}{\dd \QQ} (S_n \!:\! 0 \! \leq \! n \! \leq \! \tau_b) \right) \!&\!=\!&\! \frac{\PP(\tau_b < \infty)}{C} , \\
\label{eq:AR 1}
\QQ \left( \{S_n\} \in \cdot \, \left\lvert \ U \leq \frac{\I{\tau_b < \infty}}{C} \frac{\dd \PP}{\dd \QQ} (S_n \!:\! 0 \! \leq \! n \! \leq \! \tau_b) \right. \right) \!&\!=\!&\! \PP \left( \{S_n\} \in \cdot \, | \tau_b < \infty \right) ,
\end{eqnarray}
over events $B \in \mathcal{F}_{\tau_b}$ such that $B \subseteq \{ \tau_b<\infty \}$.

The Acceptance-Rejection procedure consists on iterating the steps: (i) sample jointly $\left(U, (S_0, \ldots, S_{\tau_b}) \right)$ from $\QQ$, and (ii) check whether 
\begin{eqnarray}\label{eq:acc event}
U & \leq & \frac{\I{\tau_b < \infty}}{C} \frac{\dd \PP}{\dd \QQ} (S_n \!:\! 0 \! \leq \! n \! \leq \! \tau_b)
\end{eqnarray}
holds.
The algorithm stops, ``accepts'', the first time inequality~\eqref{eq:acc event} is satisfied, and outputs the path $(S_0, \ldots, S_{\tau_b})$.
Equation~\eqref{eq:AR 2} states that a sample is ``accepted'' with probability $\PP(\tau_b < \infty) / C$; and equation~\eqref{eq:AR 1} assures that \guidonote{\emph{the simulation is exact}, i.e., that} the distribution of the output is $\PP(\, \cdot \, | \tau_b < \infty)$.

\subsection{\guidonote{A desirable property} for conditional sampling}

We now \guidonote{propose a criterion} for when \guidonote{a ``changed'' or ``proposal''} transition kernel $\mathsf{Q}$ is useful in sampling paths up to $\tau_b$ from the conditional distribution $\PP(\, \cdot \, | \tau_b < \infty)$.
Simply put, \guidonote{our proposed criterion} states that \guidonote{all} crossing events occur with higher probability under the \guidonote{proposal} measure than under the original.
\guidonote{Intuitively thus, this criterion is an efficiency condition for the exact conditional sampling problem.}

\begin{definition}[\guidonote{Direct proposal for exact} conditional sampling]\label{def:cond eff}
Let $\mathsf{Q}(y, \, \dd z)$ be a transition kernel such that $\mathsf{P} (y, \cdot) \ll \mathsf{Q} (y, \cdot)$ for all $y$.
Let $\QQ$ be the distribution of $\{S_n\}$ on $\R^\N$ induced by $\mathsf{Q}$.
We say that $\QQ$ is \emph{\guidonote{a direct proposal for exact} conditional sampling from $\PP (\, \cdot \, | \tau_b < \infty)$} iff
\begin{eqnarray*}
\QQ \left( \{S_n\} \in B \right) \geq \PP \left( \{S_n\} \in B \right) ,
\end{eqnarray*}
for all events $B \in \mathcal{F}_{\tau_b}$ such that $B \subseteq \{ \tau_b<\infty \}$, \guidonote{and the inequality is strict for some such $B$. Here} $\mathcal{F}_{\tau_b}$ is the usual $\sigma$-algebra associated to the stopping time $\tau_b$.
\end{definition}

We remark that the previous notion does not require $\QQ( \tau_b<\infty )=1$, although that is true for the Blanchet-Glynn change of measure, as we will see in Proposition~\ref{prop:infinite hitting time}.
We \guidonote{also} remark that by the definition of likelihood ratio we have that for all $B \subseteq \{ \tau_b < \infty \}$ it holds that $\PP \left( \{S_n\} \in B \right) = \EE^Q \left[ {\bf 1}_{B} \I{\tau_b < \infty} \cdot \dd \PP / \dd \QQ (S_n : 0 \! \leq \! n \! \leq \! \tau_b) \right]$.
Together with Definition~\ref{def:cond eff}, this identity gives the following equivalent condition \guidonote{for a proposal measure being direct for exact conditional sampling}.

\begin{corollary}\label{prop:efficiency}
The following statements are equivalent:
	\begin{enumerate}
	\item $\QQ$ is \guidonote{a direct proposal for exact} conditional sampling from $\PP (\, \cdot \, | \tau_b < \infty)$
	\item $\I{\tau_b < \infty} \cdot \dd \PP / \dd \QQ (S_n : 0 \! \leq \! n \! \leq \! \tau_b) \leq 1$ holds $\QQ$-a.s.
	\end{enumerate}
\end{corollary}

\guidonote{\paragraph{An algorithmic equivalence and motivation.}
	We now show yet another equivalent condition for a measure to be a direct proposal for exact conditional sampling.
	Informally speaking, this is an ``algorithmic'' characterization, since it states the property of being a direct proposal as a \emph{correctness} property of a simulation algorithm.
	This algorithm has been a keystone of several recent exact simulation works, see e.g.~\cite{blanchet2011exact,dong2014strong,blanchet2014exact,blanchet2012steady,liu2016optimal}, since in particular it samples a Bernoulli random variable with parameter $\PP(\tau_b < \infty)$, without the need to know the actual value of $\PP(\tau_b < \infty)$.
	This algorithmic property initially motivated the research presented in the current paper, and also motivates the terminology for calling a proposal \emph{direct} for exact conditional sampling.
}

Consider the following procedure: sample a path $(S_0, \ldots, S_{\tau_b})$ from $\QQ$; set $I := 1$ if $\dd \PP / \dd \QQ (S_n : 0 \! \leq \! n \! \leq \! \tau_b) \leq 1$, and $I:=0$ otherwise; output $\left( I, (S_0, \ldots, S_{\tau_b}) \right)$.
\guidonote{It holds that if} $\QQ(\tau_b < \infty)=1$ then parts (i) and (ii) of Corollary~\ref{prop:efficiency} are \guidonote{actually} equivalent to the following statement: $I$ is distributed as a Bernoulli random variable with parameter $\PP(\tau_b < \infty)$ and if $I=1$ then the sample path $(S_0, \ldots, S_{\tau_b})$ follows the distribution $\PP(\, \cdot \, | \tau_b < \infty)$.
Indeed, this is direct from~\eqref{eq:AR 2} and~\eqref{eq:AR 1} using $C=1$, by Corollary~\ref{prop:efficiency} part~(ii).

\subsection{The Blanchet-Glynn change of measure}

We now present the essential ideas of the Blanchet-Glynn change of measure~\cite{blanchet2008efficient}.
This measure proved efficient for estimating the probability $\PP (\tau_b < \infty)$ as $b \to \infty$.
In the current paper, we are interested in its use for the \guidonote{exact} conditional sampling problem.

The main idea motivating the Blanchet-Glynn change of measure is to approximate the transition kernel of the conditional distribution.
Indeed, it is well-known that the one-step transition kernel of $\PP (\, \cdot \, | \tau_b < \infty)$, say $\mathsf{Q}^\bast$, satisfies
\begin{eqnarray}\label{eq:def P}
\mathsf{Q}^\bast (y, \ \dd z) = \mathsf{P}(y, \ \dd z) \cdot \guidonote{\frac{\mathbb{P}_z (\tau_b < \infty)}{\mathbb{P}_y (\tau_b < \infty)}},
\end{eqnarray}
where $\mathsf{P}$ is the original transition kernel of $\{S_n\}$; see~\cite[\S VI.7]{asmussen2007stochastic}.
Here, \guidonote{the term $\mathbb{P}_y (\tau_b < \infty)$ in the denominator of~\eqref{eq:def P}} can be interpreted as a normalizing term, since \guidonote{$\int \mathsf{P}(y, \ \dd z) \mathbb{P}_z (\tau_b < \infty) = \mathbb{P}_y (\tau_b < \infty)$} for all $y$.
It is nevertheless impractical to simulate from this kernel because \guidonote{usually the values of $\mathbb{P}_z (\tau_b < \infty)$ for $z<b$ are not readily known}.
\guidonote{To be consistent with the notation of Blanchet~and Glynn~\cite{blanchet2008efficient} we denote $u^*(x) := \mathbb{P}_x (\tau_0 < \infty)$ for all $x \in \mathbb{R}$; in particular, the definition of $\mathsf{Q}^\bast$ in~\eqref{eq:def P} takes the form
\begin{eqnarray}\label{def:P star BG}
\mathsf{Q}^\bast (y, \ \dd z) = \mathsf{P}(y, \ \dd z) \cdot \frac{u^*(z-b)}{u^* (y-b)}.
\end{eqnarray}
}

The idea put forth by Blanchet and Glynn is to approximate $u^*$ using the asymptotic approximation given by Pakes-Veraverbeke Theorem, see~\cite[Chapter 5]{foss2011introduction}.
This result states that
\begin{eqnarray}\label{res:PV}
u^* (x) = \mathbb{P}_x (\tau_0 < \infty) \sim \frac{1}{|\EE X|} \int_{-x}^\infty \PP(X > s) \ \dd s \quad \text{as} \quad x \to -\infty
\end{eqnarray}
for random walks with negative drift and step sizes $X$ which are (right) \emph{strongly subexponential}.
Inspired by this fact, the Blanchet-Glynn change of measure uses the following transition kernel:
\begin{eqnarray}\label{def:Q}
\mathsf{Q}^\bc (y, \ \dd z) := \mathsf{P}(y, \ \dd z) \cdot \frac{v (z-b-c)}{w (y-b-c)} ,
\end{eqnarray}
where
\begin{eqnarray*}
v (x) &:=& \min \left\{1, \frac{1}{|\EE X|} \int_{-x}^\infty \PP(X > s) \ \dd s \right\} \label{def:v}
\end{eqnarray*}
and $w(y-b-c) := \int \mathsf{P}(y, \ \dd z) w(z-b-c) = \int \mathsf{P}(y-b, \ \dd z) v(z-b-c)$ is a normalizing term.
The constant $c \in \R$ is a translation parameter, which in~\cite{blanchet2008efficient} and in our work, we will see, is eventually chosen sufficiently large.
\guidonote{Nonetheless, a heuristic but ultimately fallacious argument for choosing $c$ large is that if we could choose ``$c=\infty$'' then by the Pakes-Veraverbeke asymptotic result~\eqref{res:PV} we would have that ``$\mathsf{Q}^\bc = \mathsf{Q}^\bast$'', i.e., the Blanchet-Glynn measure $\mathsf{Q}^\bc$ matches the conditional one step transition kernel $\mathsf{Q}^\bast$.
}

In proving \guidonote{our} results for this transition kernel we heavily rely on the fact that the functions $v$ and $w$ are closely related to the \emph{residual life tail distribution of $X$}.
That is, a random variable $Z$ with distribution given by
\begin{eqnarray}\label{def:Z}
\PP(Z>t) := \min \left\{1, \frac{1}{|\EE X|} \int_{t}^\infty \PP(X > s) \ \dd s \right\} \qquad \text{ for all $t$.}
\end{eqnarray}
We thus have $v (x)=\PP(Z>-x)$ and $w (x) = \PP(X+Z>-x)$ for all $x$\guidonote{, and in particular $\mathsf{Q}^\bc (y, \ \cdot) = \QQ^\bc (S_1-S_0 \in \cdot \ | S_0 = y) = \PP \left( X \in \cdot \ | X+Z>c+b-y \right)$, where $X$ and $Z$ are independent}.
For further details we refer the reader to~\cite{blanchet2008efficient} and~\cite[\S VI.7]{asmussen2007stochastic}.

\section{Main result: a threshold for  \guidonote{being a direct proposal}}\label{sec:main}

In this section we present our main result \guidonote{on whether the Blanchet-Glynn~\cite{blanchet2008efficient} change of measure is a direct proposal for exact conditional sampling of random walks with regularly varying step sizes}.
\guidonote{Our main result is Theorem~\ref{theo:pareto eff}, which establishes a dichotomy for the tail index of the step size distribution.}
We give two results from which our main result easily follows.
The first \guidonote{is a characterization for when the Blanchet-Glynn change of measure is a direct proposal for exact conditional sampling}, and the second explores this characterization in the case of regularly varying step sizes.
Lastly, we study the time at which the barrier is hit under the Blanchet-Glynn measure.

\paragraph{Main result.}
We now describe the main result of this paper.
We work under the following assumptions on the distribution of the step sizes, in addition to the assumption of negative drift, i.e., $\EE X < 0$.

\paragraph{Assumptions:}
\begin{itemize}
	\item[(A1)] \emph{The right tail $\PP(X^+ >\cdot)$ is regularly varying with \emph{tail index} $\alpha > 1$; that is, for all $u>0$ we have $\PP(X>u t) \sim u^{-\alpha}\PP(X>t)$ as $t \to \infty$.}
	\item[(A2)] \emph{The left tail $\PP(X^- > \cdot)$ decays fast enough so that there exists a function $h(t)=o(t)$ such that $h(t) \to \infty$ and $\int_{h(t)}^\infty \PP(X^->s) \ \dd s =o(t \cdot \PP(X^+>t))$ as $t \to \infty$.}
	\item[(A3)] \emph{The step size distribution has a continuous density which is regularly varying with tail index $\alpha+1$.}
\end{itemize}

We note that the more natural condition $\PP(X^- > t) = o (\PP(X^+>t))$ as $t \to \infty$ not necessarily implies Assumption~(A2); although it does imply that $\int_{h(t)}^\infty \PP(X^->s) \ \dd s = o\left( h(t) \cdot \PP(X^+>h(t)) \right)$ for all $h$ such that $h(t) \to \infty$.
Nonetheless, Assumption~(A2) is not overly restrictive.
Indeed, a stronger condition is that there exists $\delta>0$ such that $t^\delta \cdot\PP(X^- > t) = O (\PP(X^+>t))$ as $t \to \infty$; the latter holds for instance when $\PP(X^- > \cdot)$ is light-tailed, or when $\PP(X^- > \cdot)$ is regularly varying with tail index $\beta$ satisfying $\beta>\alpha$.
We also note that Assumption~(A3) can be replaced by the less restrictive assumption that the step size distribution be ultimately absolutely continuous with respect to Lebesgue measure, with continuous and regularly varying density.
More precisely, it can be replaced by the assumption that there exists some $t_0$ such that on $[t_0, \infty)$ the step size distribution has a continuous density $f(\cdot)$ which is regularly varying with tail index $\alpha+1$.

The main result of this paper follows.

\begin{theorem}[\guidonote{Direct proposal with} regularly varying right tails]\label{theo:pareto eff}
Let $\QQ^\bc$ be the distribution of $\{S_n\}$ induced by the transition kernel $\mathsf{Q}^\bc$ defined in~\eqref{def:Q}.
Under Assumptions~(A1)--(A3) the following hold:
\begin{enumerate}
	\item If $\alpha \in (1, \ 3/2)$ then there exists some sufficiently large $c$ so that $\QQ^\bc$ is \guidonote{a direct proposal for exact} conditional sampling from $\PP (\, \cdot \, | \tau_b < \infty)$ for all $b \geq 0$.
	\item If $\alpha \in (3/2, \ 2)$ then for all $c \in  \R$ and all $b \geq 0$ it holds that $\QQ^\bc$ is not \guidonote{a direct proposal for exact} conditional sampling from $\PP (\, \cdot \, | \tau_b < \infty)$.
\end{enumerate}
\end{theorem}

It is noteworthy that the change of measure is \guidonote{direct for exact} conditional sampling only for step sizes with very heavy tails.
Indeed, recall that the tail index $\alpha$ is an indicator of how heavy a tail is, c.f.~$\EE \left[ (X^+)^p \right] < \infty$ for $p \in (1,\alpha)$ and $\EE \left[ (X^+)^p \right] = \infty$ for $p > \alpha$.


\paragraph{Proof elements.}
We show here the main elements of the proof of Theorem~\ref{theo:pareto eff}, and start by investigating \guidonote{in Proposition~\ref{prop:efficiency characterization}} how the following statements are related.
%
The proof is deferred to Appendix~\ref{app:eff char}.
\begin{itemize}
	\item[($\text{S1}_b^c$)] The distribution $\QQ^\bc$ induced by the Blanchet-Glynn kernel~\eqref{def:Q} is \guidonote{direct for exact} for conditional sampling from $\PP(\, \cdot \, | \tau_b < \infty)$.
	\item[(S2)] We have $\PP(X+Z>t) \leq \PP(Z>t)$ for all sufficiently large $t$, where $Z$ has the residual life distribution~\eqref{def:Z} and is independent of $X$.
\end{itemize}

\begin{proposition}\label{prop:efficiency characterization}
\begin{enumerate}
	\item If (S2) holds, then there exists some sufficiently large $c$ so that ($\text{S1}_b^c$) holds for all $b \geq 0$.
	\item Suppose that $\PP(|X| \leq \delta) > 0$ for all $\delta>0$.
	If ($\text{S1}_b^c$) holds for some $b \geq 0$ and some $c \in \R$, then (S2) also holds.
\end{enumerate}
\end{proposition}

We remark that part (i) says that the same parameter $c$, chosen sufficiently large, works for \emph{all} barriers $b \geq 0$; that is, $b$ is independent of $c$ in this case.
We also remark that in the case of (ii), applying (i) we get that ($\text{S1}_b^c$) actually holds \emph{for all} $b \geq 0$, possibly after changing the constant $c$.

It is shown in~\cite{blanchet2008efficient} that $\PP(X+Z>t)-\PP(Z>t) = o\left( \PP(X>t) \right)$ as $t \to \infty$ for the family of strongly subexponential distributions, which includes regularly varying tails.
Hence, the previous proposition shows that for \guidonote{a measure to be direct for exact} conditional sampling it is not enough to know that the difference decays faster than $\PP(X> t)$, we actually need the sign of  the difference as $t \to \infty$.

The following result shows that, in the case of step sizes satisfying Assumptions~(A1)--(A2), the sign of $\PP(X+Z>t)-\PP(Z>t)$ when $t \to \infty$ is fully determined by the tail index $\alpha$ of the right tail distribution.
The proof is given in Section~\ref{sec:proof}.

\begin{theorem}\label{theo:pareto case}
\guidonote{Suppose} that Assumptions~(A1)--(A3) hold.
Let $Z$ be a random variable independent of $X$ with the residual life distribution~\eqref{def:Z}.
Then the following statements hold:
\begin{enumerate}
\item If $\alpha \in (1,\ 3/2)$ then $\PP(X+Z>t) \leq \PP(Z>t)$ for all $t>0$ sufficiently large.
\item If $\alpha \in (3/2,\ 2)$ then $\PP(X+Z>t) \geq \PP(Z>t)$ for all $t>0$ sufficiently large.
\end{enumerate}
\end{theorem}

With this, Theorem~\ref{theo:pareto eff} is a corollary of Proposition~\ref{prop:efficiency characterization} and Theorem~\ref{theo:pareto case}.


\paragraph{Hitting time analysis.}
We now investigate the finiteness and mean value of the hitting time $\tau_b$ under the Blanchet-Glynn change of measure $\QQ^\bc$.
The motivation is that $\tau_b$ gives a rough estimate of the computational effort of sampling a barrier-crossing path using the measure $\QQ^\bc$.

The following result explores the hitting time in the regularly varying right tails setting of Assumptions~(A1)--(A2).
Its proof is deferred to Appendix~\ref{app:hitting}.

\begin{proposition}[Hitting time under $\QQ^\bc$]\label{prop:infinite hitting time}
Let $\QQ$ be the distribution of $\{S_n\}$ induced by the transition kernel $\mathsf{Q}^\bc$ defined in~\eqref{def:Q}.
Consider the setting of Assumptions~(A1)--(A3).
For any sufficiently large $c$ the following hold for all $b \geq 0$:
\begin{enumerate}
	\item If $\alpha > 1$ then $\QQ^\bc(\tau_b < \infty) = 1$.
	\item If $\alpha \in (1,3/2)$ then $\EE^{\QQ^\bc} \tau_b = \infty$ for all $b \geq 0$.
	\item If $\alpha > 2$ then $\EE^{\QQ^\bc} \tau_b = O(b)$ as $b \to \infty$.
\end{enumerate}
\end{proposition}

We remark that part (ii) of the previous proposition, although a negative result, is actually independent of the Blanchet-Glynn measure $\QQ^\bc$ and holds essentially because we have $\EE^{\PP}[\tau_b | \tau_b < \infty] = \infty$ when $\alpha \in (1,2)$, see e.g.~\cite[Theorem~1.1]{asmussen1996large}.
In other words, if $\alpha \in (1,2)$ no algorithm \guidonote{or change of measure} --- \guidonote{direct} or not --- sampling paths $(S_0, \ldots, S_{\tau_b})$ from $\PP (\, \cdot \, | \tau_b < \infty)$ can produce paths of finite expected length.

We also remark that part (ii) of Proposition~\ref{prop:infinite hitting time} is a counterexample for Proposition~4 of~\cite{blanchet2008efficient}.
Indeed, the latter result claims that we have $\EE^{\QQ^\bc} \tau_b = O(b)$ as $b \to \infty$ when the step sizes satisfy $\EE^\PP [X^p; \, X>0]<\infty$ for some $p>1$.
Part (ii) of Proposition~\ref{prop:infinite hitting time} shows that the latter condition is not enough in general.
Alternatively, this issue with Proposition~4 of~\cite{blanchet2008efficient} can also be derived from Corollary~1 of~\cite{blanchet2012efficient}.
The latter result shows that if $\alpha \in (1, 3/2)$ no change of measure can be at the same time \emph{strongly efficient for importance sampling} and have linear expected hitting time; in contrast, Proposition~4 of~\cite{blanchet2008efficient} states that for any $\alpha>1$ the Blanchet-Glynn measure is both \emph{strongly efficient for importance sampling} and has linear expected hitting time.
Clearly both results are contradictory.

\section{Threshold 3/2: a comparison}\label{sec:comparison}


In this section we compare the threshold result of Theorem~\ref{theo:pareto eff} with previous simulation works where, when using regularly varying step sizes, some form of efficiency of the method has given rise to the same threshold 3/2 for the tail index.
We argue that the threshold arises in this existing literature for reasons unrelated to our work.

\paragraph{Review.} Previous works in which the $3/2$ threshold appears in the context of efficiency are Blanchet and Liu~\cite{blanchet2012efficient} and Murthy, Juneja and Blanchet~\cite{murthy2014state}.
Both papers focus on solving the probability estimation problem via importance sampling; that is, their aim is to estimate the probability $\PP( \tau_b < \infty )$ for arbitrarily large barriers $b$, using Monte Carlo sampling from another measure.
Blanchet and Liu propose a parameterized and state-dependent change of measure, say $\QQ^{\mathsf{BL}}$, which in the regularly varying case takes the form of a mixture between a \emph{big-} and a \emph{small-jump} transition kernel.
Murthy \emph{et al.}~propose a similar big- and small-jump mixture kernel, say $\QQ^{\mathsf{MJB}}$, however their change of measure is state-independent and additionally conditions on the time interval at which the barrier-crossing event occurs.

Both Blanchet and Liu~\cite{blanchet2012efficient} and Murthy \emph{et al.}~\cite{murthy2014state} have two requirements on their proposed measures: (i) the linear scaling $\EE^\QQ \tau_b = O(b)$ as $b \to \infty$ of the hitting time, and (ii) \emph{strong efficiency} of the estimation procedure.
In short, the latter means that, under the proposed change of measure $\QQ$, the coefficient of variation of the random variable
\begin{eqnarray}\label{eq:LR eff}
\frac{\dd\PP}{\dd\QQ} (S_n : 0 \!\leq\! n \!\leq\! \tau_b) \cdot \I{\tau_b < \infty}
\end{eqnarray}
stays bounded as $b \to \infty$; this is a second moment condition on~\eqref{eq:LR eff}.
Both papers arrive at the same threshold result: for regularly varying step sizes, the proposed change of measure satisfies the previous two requirements for some combination of tuning parameters if and only if the tail index $\alpha$ is greater than 3/2.

\paragraph{Comparison.} Given that the same threshold appears, it is natural to ask if there is a connection between our result in Corollary~\ref{theo:pareto eff} and the results in prior work.
We now argue why there is no clear or direct connection between these results.

In Blanchet and Liu~\cite{blanchet2012efficient} and Murthy \emph{et al.}~\cite{murthy2014state} the threshold 3/2 is strictly related to the second moment condition over the likelihood ratio~\eqref{eq:LR eff} that is imposed by the requirement of efficiency for importance sampling.
More precisely, in both these works if a moment condition is imposed on a different moment than the second, then we get a different threshold for the admissible tail indexes. 
In contrast, our result arises from imposing an almost sure condition on the Blanchet-Glynn change of measure. 
Indeed, by Corollary~\ref{prop:efficiency}, \guidonote{the condition of a proposal measure being direct for exact} conditional sampling is a $\QQ$-almost sure condition on the random variable~\eqref{eq:LR eff}.
In contrast, and as said before, efficiency for importance sampling is a second moment condition on~\eqref{eq:LR eff}.

\section{Proof of Theorem~\ref{theo:pareto case}}\label{sec:proof}

In this section we prove Theorem~\ref{theo:pareto case}, which is the main component of our main result of Theorem~\ref{theo:pareto eff}.
That is, we prove that the tail index $\alpha$ completely determines the sign of the difference
\begin{eqnarray}\label{proof:difference}
\PP(X+Z>t)-\PP(Z>t)
\end{eqnarray}
when $t$ is large enough.
We work under Assumptions~(A1)--(A3), which state roughly speaking that the step sizes have regularly varying right tails with tail index $\alpha$, and lighter left tails.
In short, Theorem~\ref{theo:pareto case} establishes that if $\alpha \in (1, \, 3/2)$ then the difference~\eqref{proof:difference} is negative for large $t$, and positive if $\alpha \in (3/2, \, 2)$.

The following is a roadmap for the main steps of the proof.
First, in Lemma~\ref{lemma:decomposition} we write the difference~\eqref{proof:difference} as a sum of several terms.
Second, in Lemma~\ref{lemma:asymptotics} we carry out an asymptotic analysis to determine which terms dominate when $t \to \infty$.
It follows that the sign of the difference~\eqref{proof:difference} when $t \to \infty$ can be reduced to the sign of the sum of dominant terms when $t \to \infty$.
Finally, the latter is analyzed in Lemma~\ref{lemma:asymptotics sign}, which reveals the dichotomy for $\alpha$ in $(1, \, 3/2)$ or $(3/2, \, 2)$.

Before embarking on the proof, some remarks on our notation are in order.
Recall that we say that the random variable $Z$ has the \emph{residual life distribution of $X$} if its distribution is given by
\begin{eqnarray*}
\PP(Z>t) &=& \min \left\{1, \, \frac{1}{| \EE X |} \int_{t}^\infty \PP(X > s) \ \dd s  \right\}, \quad \text{for all } t.
\end{eqnarray*}
We write the left-most point of the support of $Z$ as $z_0 := \inf \{ t \ : \ \PP(Z>t)<1 \}$, which is finite since $\EE X$ is also finite.
Additionally, we use that the density of $Z$ is $\PP(X>t) / |\EE X|$ for all $t > z_0$ and that $\int_{z_0}^\infty \PP(X > s) \ \dd s = | \EE X |$.
We also use the notation \guidonote{$F(t) := \PP(X \leq t)$ and} $\overline{F}(t) := \PP(X > t)$ for all $t$.
Lastly, we recall that Assumption~(A1) establishes that the right tail $\PP(X>\cdot)$ is regularly varying with \emph{tail index} $\alpha > 1$.

We start with a general decomposition of the difference~\eqref{proof:difference}.

\begin{lemma}\label{lemma:decomposition}
Let $X$ be a random variable with negative mean, and let $Z$ be independent of $X$ with the residual life distribution of $X$.
Consider any function $h$ such that $\max\{z_0, 0\} < h(t) < t/2$ for all $t > \max\{2 z_0, 0\}$.
Then the following holds for $t > \max\{2 z_0, 0\}$:
\begin{eqnarray*}
\PP(X+Z>t)-\PP(Z>t) &=& \ppp (t) - \qqq (t) + \eo(t) - \et(t),
\end{eqnarray*}
where we define for $t > \max\{2 z_0, 0\}$
\begin{eqnarray*}
\ppp (t) &:=& \frac{1}{|\EE X|} \int_{h(t)}^{t-h(t)} \overline{F}(t-s) \cdot \overline{F}(s) \ \dd s \\
\qqq (t) &:=& \frac{\overline{F}(t)}{|\EE X|} \int_{h(t)}^\infty \left[ 2\overline{F}(s)-F(-s) \right] \ \dd s \\
\eo(t)&:=& \frac{1}{|\EE X|} \left[ \left( \int_0^{h(t)} + \int_{z_0}^{h(t)} \right) \left[ \overline{F}(t-s)-\overline{F}(t) \right] \cdot \overline{F}(s) \ \dd s \right. \nonumber \\
&& \qquad\qquad\qquad \left. +\int_{-h(t)}^0 \left[ \overline{F}(t) - \overline{F}(t-s) \right] \cdot F(s) \ \dd s \right] \\
\et(t) &:=& \frac{1}{|\EE X|} \int_{-\infty}^{-h(t)} \overline{F}(t-s) \cdot F(s) \ \dd s .
\end{eqnarray*}
\end{lemma}

\begin{proof}
First note that
$$\PP(X+Z>t)-\PP(Z>t) = \PP(X+Z>t, Z \leq t)-\PP(X+Z \leq t, Z>t).$$
For $t$ satisfying $\max\{z_0,0\} < h(t) < t/2$ we decompose the first term on the right-hand side as follows:
\begin{eqnarray*}
\lefteqn{\PP(X+Z>t, Z \leq t) = \int_{z_0}^{t} \overline{F}(t-s) \cdot \frac{1}{|\EE X|} \overline{F}(s) \ \dd s } \\
&=& \frac{1}{|\EE X|} \left[ \int_{z_0}^{h(t)} \overline{F}(t) \overline{F}(s) \ \dd s + \int_{z_0}^{h(t)} \left[ \overline{F}(t-s)-\overline{F}(t)\right] \overline{F}(s) \ \dd s \right. \\
&& \qquad + \int_{h(t)}^{t-h(t)} \overline{F}(t-s) \overline{F}(s) \ \dd s \\
&& \qquad \left. + \int_{0}^{h(t)} \overline{F}(t) \overline{F}(s) \ \dd s + \int_{0}^{h(t)} \left[\overline{F}(t-s)-\overline{F}(t)\right] \overline{F}(s) \ \dd s \right] .
\end{eqnarray*}
A similar decomposition follows for the second term:
\begin{eqnarray*}
\lefteqn{ \PP(X+Z \leq t, Z > t) = \int_{t}^{\infty} F( t-s) \cdot \frac{1}{|\EE X|} \overline{F}(s) \ \dd s }\\
&=& \frac{1}{|\EE X|} \left[ \int_{-\infty}^{-h(t)} F(s) \overline{F}(t-s) \ \dd s + \int_{-h(t)}^{0} F(s) \overline{F}(t) \ \dd s \right. \\
&& \qquad \left. + \int_{-h(t)}^{0} F(s) \left[ \overline{F}(t-s)-\overline{F}(t)\right] \ \dd s \right].
\end{eqnarray*}
Subtracting both terms we obtain
\begin{eqnarray*}
\lefteqn{ |\EE X| \cdot \left[ \PP(X+Z>t)-\PP(Z>t) \right] = } \\
&=& \int_{h(t)}^{t-h(t)} \overline{F}(t-s) \overline{F}(s) \ \dd s \\
&& - \overline{F}(t) \left[  \int_{-h(t)}^{0} F(s) \ \dd s - \int_{z_0}^{h(t)} \overline{F}(s) \ \dd s - \int_{0}^{h(t)} \overline{F}(s) \ \dd s \right] \\
&& + \left[ \left( \int_{z_0}^{h(t)} + \int_{0}^{h(t)} \right) \left[ \overline{F}(t-s)-\overline{F}(t) \right] \overline{F}(s) \ \dd s \right. \\
&& \left. - \int_{-h(t)}^{0} \left[ \overline{F}(t-s)-\overline{F}(t)\right] F(s) \ \dd s \right] - \int_{-\infty}^{-h(t)} F(s) \overline{F}(t-s) \ \dd s \\
&=& |\EE X| \cdot \left[ \ppp (t) - \qqq (t)+ \eo(t) - \et(t) \right].
\end{eqnarray*}
The last equality comes from using the definition of $\ppp$, $\qqq$, $\eo$ and $\et$, and noting that $|\EE X| = \int_{z_0}^\infty \overline{F}(s) \ \dd s$ and $\EE X < 0$, so we have that
\begin{eqnarray*}
\lefteqn{\int_{-h(t)}^{0} F(s) \ \dd s - \int_{z_0}^{h(t)} \overline{F}(s) \ \dd s - \int_{0}^{h(t)} \overline{F}(s) \ \dd s} \\
&& = \EE X^- - |\EE X| - \EE X^+ + 2\int_{h(t)}^\infty \overline{F}(s) \ \dd s - \int_{-\infty}^{-h(t)} F(s) \ \dd s \\
&& = \int_{h(t)}^\infty \left[ 2 \overline{F}(s)- F(-s)  \right] \ \dd s .
\end{eqnarray*}
This concludes the proof.
\end{proof}

The next step consists in determining which terms dominate when $t \to \infty$; this is done in the following result.

\begin{lemma}\label{lemma:asymptotics}
Let $X$ be a random variable with negative mean satisfying Assumptions~(A1)--(A3) for some index of regular variation $\alpha \in (1, 2)$.
Let $Z$ be independent of $X$ with the residual life distribution of $X$.
In the definition of $p$, $q$, $\eo$ and $\et$ consider a function $h$ satisfying Assumption~(A2); in particular $h$ satisfies $\max\{z_0, 0 \} < h(t) < t/2$ for all $t > \max\{2 z_0, 0\}$, and it holds that $h(t) \to \infty$ and $h(t) = o \left( t \right)$ as $t \to \infty$.
Then the following hold as $t \to \infty$:
\begin{enumerate}
\item $$\ppp(t)-\qqq(t) \sim \frac{t \overline F(t)^2}{|\EE X|} \left( 2\int_0^{1/2}[(1-u)^{-\alpha}-1] u^{-\alpha}\ \dd u - \frac{2^{\alpha}}{\alpha-1} \right),$$
\item $$\eo(t) = o \left( t \overline F(t)^2 \right) \qquad \text{and} \qquad \et(t) = o \left( t \overline F(t)^2 \right).$$
\end{enumerate}
\end{lemma}

We remark that Lemmas~\ref{lemma:decomposition} and \ref{lemma:asymptotics} together establish that if $X$ satisfies Assumptions~(A1)--(A3) and $\alpha \in (1,2)$ then
\begin{eqnarray}\label{sim:prop 3 BG}
\PP(X+Z>t) - \PP(Z>t) &\sim& K_\alpha \PP(Z>t) \PP(X>t)
\end{eqnarray}
as $t \to \infty$, where $$K_\alpha := (\alpha-1) \int_0^1 \left( (1-u)^{-\alpha}-1 \right) \left( u^{-\alpha}-1 \right) \dd u - (\alpha+1).$$
This comes from $\PP(Z>t) \sim t \overline F(t) / ((\alpha-1)|\EE X|)$ by Karamata's Theorem~\cite[Theorem 1.6.1]{bingham1989regular}.
We note that, in contrast, Proposition~3 of~\cite{blanchet2008efficient} shows that $\PP(X+Z>t) - \PP(Z>t) = o \left( \PP(X>t) \right)$, so the result~\eqref{sim:prop 3 BG} is much finer.

\begin{proof}
We start proving (i).
For that, first rewrite $|\EE X| \cdot (p(t)-q(t)) / (t \overline F(t)^2)$ as
$$2\int_{h(t)}^{t/2} \frac{\overline F(t-s)-\overline F(t)}{\overline F(t)} \cdot \frac{\overline F(s)}{t \overline F(t)} \ \dd s - \frac{2 \int_{t/2}^\infty \overline F(s) \ \dd s}{t \overline F(t)} + \frac{\int_{h(t)}^\infty F(-s) \ \dd s}{t \overline F(t)}.$$
The third term goes to zero by Assumption~(A2), so we can ignore it for the proof of the statement.
For the second term, note that since $\alpha>1$, Karamata's Theorem~\cite[Theorem 1.6.1]{bingham1989regular} yields
$$2\int_{t/2}^\infty \overline F(s) \ \dd s \sim \frac{t \overline F(t/2)}{\alpha-1} \sim \frac{2^\alpha}{\alpha-1} t\overline F(t).$$
It remains to investigate the first term.
To this end, we first rewrite the integral as
\begin{eqnarray}\label{eq:integral i}
\int_{h(t)/t}^{1/2} \left[\frac{\overline F(t (1-u))}{\overline F(t)} - 1\right] \frac{\overline F(t u)}{\overline F(t)} \ \dd u;
\end{eqnarray}
we need to show that as $t \to \infty$ this integral converges to $\int_0^{1/2} [(1-u)^{-\alpha}-1] u^{-\alpha}\ \dd u$.
To this end, consider $\delta \in (0, 1/2)$ and note that since $h(t)=o(t)$ we can write~\eqref{eq:integral i} for all sufficiently large $t$ as
\begin{eqnarray}\label{eq:integral ii}
\int_{h(t)/t}^{\delta} \left[\frac{\overline F(t (1-u))}{\overline F(t)} - 1\right] \frac{\overline F(t u)}{\overline F(t)} \ \dd u
+
\int_{\delta}^{1/2} \left[\frac{\overline F(t (1-u))}{\overline F(t)} - 1\right] \frac{\overline F(t u)}{\overline F(t)} \ \dd u.
\end{eqnarray}
We start by analyzing the second term in~\eqref{eq:integral ii}.
Since $\overline F$ is regularly varying with tail index $\alpha$ we get by the Uniform Convergence Theorem~\cite[Theorem 1.2.1]{bingham1989regular} that
$$\sup_{u \in [\delta, \, 1/2]} \left\lvert \left[\frac{\overline F(t (1-u))}{\overline F(t)} - 1\right] \frac{\overline F(t u)}{\overline F(t)} - [(1-u)^{-\alpha}-1] u^{-\alpha} \right\rvert \to 0 \quad \text{as } t \to \infty,$$
so
\begin{eqnarray}\label{lim lemma}
\int_{\delta}^{1/2} \left[\frac{\overline F(t (1-u))}{\overline F(t)} - 1\right] \frac{\overline F(t u)}{\overline F(t)} \ \dd u \to \int_\delta^{1/2} [(1-u)^{-\alpha}-1] u^{-\alpha}\ \dd u,
\end{eqnarray}
as $t \to \infty$.
We next analyze the first term in~\eqref{eq:integral ii}.
Using Assumption~(A3) we apply the mean value theorem on the interval $(0, u)$ to the function $s \mapsto \overline F(t (1-s))$ and establish  that $\overline F(t (1-u))/\overline F(t) -1 = f \left(t(1-\xi)\right) t u / \overline F(t)$ for some $\xi=\xi(t, u) \in (0, u)$, where $f$ is the density of $X$.
Additionally, we have that for all sufficiently large $t$ it holds that $f \left(t(1-\xi)\right) t/ \overline F(t) \leq 2(1+2^{\alpha+1})\alpha$ for all $\xi \in (0,\delta)$.
Indeed, since $f$ is regularly varying with tail index $\alpha+1$  then by Karamata's Theorem~\cite[Theorem 1.6.1]{bingham1989regular} we get $t f(t)/\overline F(t) \to \alpha$; additionally, by Uniform Convergence Theorem~\cite[Theorem 1.2.1]{bingham1989regular} we have that for any large enough $t$
$$\sup_{\xi \in (0, \delta)} \left\lvert \frac{f(t(1-\xi))}{f(t)} - \frac{1}{(1-\xi)^{\alpha+1}} \right\rvert \leq 1,$$
so $f(t(1-\xi)) / f(t) \leq 1+1/(1-\xi)^{\alpha+1} \leq 1+2^{\alpha+1}$ for all sufficiently large $t$ and for all $\xi \in (0, \delta) \subset (0, 1/2)$.
We conclude that $\overline F(t (1-u))/\overline F(t) -1 \leq 2(1+2^{\alpha+1})\alpha \guidonote{u}$ for all large enough $t$.
We use this inequality to bound the term in the brackets of the first term of~\eqref{eq:integral ii}, obtaining that for all sufficiently large $t$
$$\int_{h(t)/t}^{\delta} \left[\frac{\overline F(t (1-u))}{\overline F(t)} - 1\right] \frac{\overline F(t u)}{\overline F(t)} \ \dd u \leq 2(1+2^{\alpha+1})\alpha \int_{h(t)/t}^{\delta} \frac{t u \overline F(t u)}{t \overline F(t)} \ \dd u.$$
We now argue that $\int_{h(t)/t}^{\delta} \left( t u \overline F(t u) \right) / \left( t \overline F(t) \right) \ \dd u \leq 2\delta^{2-\alpha}/(2-\alpha)$ for all sufficiently large $t$.
Indeed,
$$\int_{h(t)/t}^{\delta} \frac{ tu \overline F(t u) }{ t \overline F(t) } \ \dd u = \left(\int_0^{\delta t} - \int_0^{h(t)}\right) \frac{ s \overline F(s) }{ t^2 \overline F(t) } \ \dd s,$$
so using that $t \overline F(t)$ is regularly varying with tail index $\alpha-1 \in (0, 1)$ we can apply Karamata's Theorem~\cite[Theorem 1.6.1]{bingham1989regular} and that $h(t) \to \infty$ to get
$$\int_{h(t)/t}^{\delta} \frac{tu \overline F(t u)}{t \overline F(t)} \ \dd u \sim \frac{1}{2-\alpha} \left( \delta^2 \frac{\overline F(\delta t)}{\overline F(t)} - \left(\frac{h(t)}{t}\right)^2 \frac{\overline F(h(t))}{\overline F(t)} \right) \sim \frac{\delta^{2-\alpha}}{2-\alpha}$$
as $t\to\infty$.
Indeed, note that since $h(t)=o(t)$ and $s \mapsto s^2 \overline F(s)$ is regularly varying with tail index $\alpha-2 \in (-1,0)$ then $\left(h(t) / t \right)^2 (\overline F(h(t)) / \overline F(t)) \to 0$.
All in all, we obtain that the first term of~\eqref{eq:integral ii} satisfies, for all large enough $t$,
$$\int_{h(t)/t}^{\delta} \left[\frac{\overline F(t (1-u))}{\overline F(t)} - 1\right] \frac{\overline F(t u)}{\overline F(t)} \ \dd u \leq \frac{4(1+2^{\alpha+1})\alpha}{2-\alpha} \delta^{2-\alpha}.$$
Lastly, note that $\delta \in (0, 1/2)$ is arbitrary, so letting $\delta$ decrease to 0 in the latter inequality we get that 
\begin{eqnarray}\label{eq:limit i}
\lim_{\delta \searrow 0} \ \limsup_{t \to \infty} \ \int_{h(t)/t}^{\delta} \left[\frac{\overline F(t (1-u))}{\overline F(t)} - 1\right] \frac{\overline F(t u)}{\overline F(t)} \ \dd u = 0.
\end{eqnarray}
Similarly, letting $\delta$ decrease to 0 in~\eqref{lim lemma} we obtain
\begin{eqnarray}\label{eq:limit ii}
\lim_{\delta \searrow 0} \, \lim_{t \to \infty} \int_{\delta}^{1/2} \left[\frac{\overline F(t (1-u))}{\overline F(t)} - 1\right] \frac{\overline F(t u)}{\overline F(t)} \, \dd u = \int_0^{1/2} [(1-u)^{-\alpha}-1] u^{-\alpha}\, \dd u.
\end{eqnarray}
From~\eqref{eq:limit i} and~\eqref{eq:limit ii} and the decomposition~\eqref{eq:integral ii} of~\eqref{eq:integral i} we get the desired result.


We now prove (ii).
First we show that $\eo(t) = o \left( t \overline F(t)^2 \right)$.
To this end, it is sufficient to prove
$\int_0^{h(t)} \left[ \overline{F}(t-s)-\overline{F}(t) \right] \overline{F}(s) \ \dd s = o \left( t \overline F(t)^2 \right)$.
\guidonote{Indeed, if $z_0 < 0$ it holds that $\int_{z_0}^{h(t)} \left[ \overline{F}(t-s)-\overline{F}(t) \right] \overline{F}(s) \ \dd s = o \left( t \overline F(t)^2 \right)$ as $t \to \infty$, since
\begin{eqnarray*}
\left| \int_{z_0}^0 \left[ \overline{F}(t-s)-\overline{F}(t) \right] \overline{F}(s) \ \dd s \right|
\leq \int_{0}^{|z_0|} \left[ \overline{F}(t) - \overline{F}(t+s) \right] \ \dd s
\leq |z_0| \sup_{s \in [t, t+|z_0|]} f(s) \sim |z_0| f(t)
\end{eqnarray*}
by the Uniform Convergence Theorem~\cite[Chapter 2]{foss2011introduction} and because $f$ is in particular long tailed by Assumption~(A3); and since $\alpha \in (1,2)$ then $f(t) = o \left( t \overline F(t)^2 \right)$.
}
\guidonote{It is sufficient then} to prove that the expression
$$\int_0^{h(t)} \frac{\overline{F}(t-s)-\overline{F}(t)}{\overline F(t)} \frac{\overline{F}(s)}{t \overline F (t)} \ \dd s = \int_0^{h(t)/t} \left[ \frac{\overline{F}(t(1-u))}{\overline F(t)}-1 \right] \frac{\overline{F}(t u)}{\overline F (t)} \ \dd u$$
goes to zero as $t \to \infty$.
We proceed by using the same line of reasoning used to prove~\eqref{eq:limit i}, which is delineated in the following.
First, apply the mean value theorem on the interval $(0,u)$ to the function $s \mapsto \overline F(t (1-s))$, and then use the Uniform Convergence Theorem~\cite[Theorem 1.2.1]{bingham1989regular} to get that for all sufficiently large $t$ we have
\begin{eqnarray}\label{eq:limit i'}
\int_0^{h(t)/t} \left[\frac{\overline F(t (1-u))}{\overline F(t)} - 1\right] \frac{\overline F(t u)}{\overline F(t)} \ \dd u \leq 2(1+2^{\alpha+1})\alpha \int_0^{h(t)/t} \frac{t u \overline F(t u)}{t \overline F(t)} \ \dd u.
\end{eqnarray}
Second, apply Karamata's Theorem~\cite[Theorem 1.6.1]{bingham1989regular} to obtain that
\begin{eqnarray}\label{eq:limit ii'}
\int_0^{h(t)/t} \frac{t u \overline F(t u)}{t \overline F(t)} \ \dd u = \int_0^{h(t)} \frac{s \overline F(s)}{t^2 \overline F(t)} \ \dd s \sim \frac{1}{2-\alpha} \left( \frac{h(t)}{t} \right)^2 \frac{\overline F(h(t))}{\overline F(t)},
\end{eqnarray}
since $h(t) \to \infty$.
Third, since the function $s \mapsto s^2 \overline F(s)$ is regularly varying with $2-\alpha \in (0,1)$ and $h(t)=o(t)$ then $\left( h(t)/t \right)^2 \overline F(h(t)) / \overline F(t) \to 0$ when $t \to \infty$.
The latter fact, together with~\eqref{eq:limit i'} and~\eqref{eq:limit ii'}, allows to conclude the desired result.

Lastly, $\et(t) = o \left( t \overline F(t)^2 \right)$ also holds because
$$\int_{-\infty}^{-h(t)} \frac{\overline{F}(t-s)}{\overline F(t)} \frac{F(s)}{t \overline F(t)} \ \dd s = \int_{h(t)}^\infty \frac{\overline{F}(t+s)}{\overline F(t)} \frac{F(-s)}{t \overline F(t)} \ \dd s \leq \int_{h(t)}^\infty \frac{F(-s)}{t \overline F(t)} \ \dd s,$$
with the last term going to zero as $t \to \infty$ by Assumption~(A2).
\end{proof}


The previous result shows that when $t \to \infty$ the sign of the difference $\PP(X+Z>t)-\PP(Z>t)$ reduces to the sign of the term $\ppp (t) - \qqq (t)$.
We now show that the latter is fully determined by the tail index $\alpha$ being either in $(1,\, 3/2)$ or in $(3/2,\, 2)$.


\begin{lemma}\label{lemma:asymptotics sign}
The quantity
\begin{eqnarray}\label{int:sign}
\int_0^{1/2}[(1-u)^{-\alpha}-1] u^{-\alpha}\ \dd u - \frac{2^{\alpha-1}}{\alpha-1}
\end{eqnarray}
is negative for $\alpha\in(1, \, 3/2)$ and positive for $\alpha\in(3/2, \, 2)$.
\end{lemma}


\begin{proof}
First note that for fixed $u \in (0,1/2)$ the function $[(1-u)^{-\alpha}-1] u^{-\alpha}$ is strictly increasing in $\alpha>0$, so $\int_0^{1/2}[(1-u)^{-\alpha}-1] u^{-\alpha}\ \dd u$ is as well.
Also, since $2^\beta/\beta$ is strictly-decreasing for $\beta \in (0, 1)$ then $-2^{\alpha-1} / (\alpha-1)$ is strictly increasing in $\alpha$ when $\alpha \in (1,2)$.
Thus~\eqref{int:sign} is strictly increasing in $\alpha$ for $\alpha \in (1,2)$.
Lastly, it is easy to verify that if $\alpha=3/2$ then~\eqref{int:sign} is equal to zero.
\end{proof}


With the previous lemmas the proof of Theorem~\ref{theo:pareto case} is straightforward.

\begin{proof}[Proof of Theorem~\ref{theo:pareto case}]
Consider in the definition of $p$, $q$, $\eo$ and $\et$ a function $h$ satisfying Assumption~(A2); in particular it satisfies $\max\{z_0, 0 \} < h(t) < t/2$ for all $t > \max\{2 z_0, 0\}$, and $h(t) \to \infty$ and $h(t) = o \left( t \right)$ as $t \to \infty$.
By Lemma~\ref{lemma:decomposition} and Lemma~\ref{lemma:asymptotics} we have that as $t \to \infty$
\begin{eqnarray*}
\frac{\PP(X+Z>t)-\PP(Z>t)}{t \overline F(t)^2} = \frac{\ppp (t) - \qqq (t)}{t \overline F(t)^2} + o\left( 1 \right).
\end{eqnarray*}
We conclude from Lemma~\ref{lemma:asymptotics sign} that as $t\to\infty$ the right-hand side is negative for $\alpha\in(1, \, 3/2)$ and positive for $\alpha\in(3/2, \, 2)$.
\end{proof}

\section*{Acknowledgments}

We gratefully acknowledge support form NSF under grant CMMI-1252878.
G.L.~acknowledges support from the program ``Becas de Doctorado en el Extranjero - Becas Chile - CONICYT'' \guidonote{under grant 72110679}.
\guidonote{We thank the School of Industrial and Systems Engineering at Georgia Institute of Technology, where part of this work was carried out.}
We also thank Jos\'e Blanchet for valuable discussions.

\begin{appendices}

\section{Proof of Proposition~\ref{prop:efficiency characterization}}\label{app:eff char}

\begin{proof}[Proof of Proposition~\ref{prop:efficiency characterization}]
For (i), assume that (S2) holds; that is, that we have $\PP(X+Z>t) \leq \PP(Z>t)$ for all $t$ sufficiently large.
Take then $c$ sufficiently large so that $\PP(X+Z>c+t) \leq \PP(Z>c+t)$ holds for all $t \geq 0$.
Thus, we have by definition of $v$ and $w$ that
\begin{eqnarray}\label{eq:proof eff 1}
w (y-b-c) / v (y-b-c) \leq 1\qquad \text{for all $y \leq b$.}
\end{eqnarray}
Using then the definition~\eqref{def:Q} of $\mathsf{Q}^\bc$, the following holds for $S_0=0$ and all $b \geq 0$:
\begin{eqnarray*}
\I{\tau_b < \infty} \frac{\dd \PP}{\dd \QQ^\bc} (S_n : 0 \! \leq \! n \! \leq \! \tau_b) &=& \I{\tau_b < \infty} \frac{w (S_0 - b-c)}{v (S_{\tau_b} - b-c)} \prod_{n=1}^{\tau_b-1} \frac{w (S_n - b-c)}{v (S_n - b-c)} \\
&\leq& \I{\tau_b < \infty} \frac{w (S_0 - b-c)}{v (S_{\tau_b} - b-c)}.
\end{eqnarray*}
It follows that, conditional on $\tau_b<\infty$, we have $w (S_0 - b-c) \leq v (S_0 - b-c) \leq v (S_{\tau_b} - b-c)$, by inequality~\eqref{eq:proof eff 1} and monotonicity of $v$.
So then we obtain that $\I{\tau_b < \infty} \dd \PP/\dd \QQ^\bc (S_n : 0 \! \leq \! n \! \leq \! \tau_b) \leq 1$.
We conclude that statement~($\text{S1}^\bc$) holds, by Corollary~\ref{prop:efficiency}.


For (ii), assume that~($\text{S1}^\bc$) holds, for some $b \geq 0$ and some $c \in \R$; i.e., that $\QQ^\bc$ is \guidonote{a direct proposal} for conditional sampling from $\PP(\, \cdot \, | \tau_b < \infty)$.
We proceed by contradiction and assume that (S2) does not hold, i.e., that for all $t$ we have that there exists a $t_0>t$ such that $\PP(X+Z>t_0) > \PP(Z>t_0)$.
Using the fact that $w (y) = \PP(X+Z>-y)$ and $v (y) = \PP(Z>-y)$, we get that the previous hypothesis implies in particular that for all $y \leq b $ there exists $y_0 < y$ such that $w (y_0-b-c) / v (y_0-b-c) > 1$ holds.

With this, we will show that necessarily the following holds
\begin{eqnarray}\label{ineq:contrad 1}
\QQ^\bc \left( \I{\tau_b < \infty} \frac{w (S_0 - b-c)}{v (S_{\tau_b} - b-c)} \prod_{n=1}^{\tau_b-1} \frac{w (S_n - b-c)}{v (S_n - b-c)} >1 \right) > 0,
\end{eqnarray}
i.e., that $\QQ^\bc \left( \I{\tau_b < \infty} \cdot \dd \PP/\dd \QQ^\bc (S_n : 0 \! \leq \! n \! \leq \! \tau_b) >1 \right) > 0$.
The latter is a contradiction with hypothesis ($\text{S1}^\bc$), by Corollary~\ref{prop:efficiency}.
Now, to prove~\eqref{ineq:contrad 1} the main idea is to construct paths $(S_n : 0 \!\leq\! n \!\leq\! \tau_b)$ under $\QQ^\bc$ that, before crossing the barrier $b$, spend a sufficiently large amount of time in the set
\begin{eqnarray*}
Y_{>1}^{c,b} := \left\{ y \leq b \ : \ \frac{ w (y-b-c) }{ v (y-b-c) } > 1 \right\}.
\end{eqnarray*}
For that we distinguish two cases: if $S_0=0 \in Y_{>1}^{c,b}$ or not.
We start with the former case.

\paragraph{Case 1: if $c$ and $b \geq 0$ are such that $S_0 = 0 \in Y_{>1}^{c,b}$.}
In this case we have that for all $C>0$ there exists $N>0$ such that
$$\PP\left( \prod_{n=1}^{N-1} \frac{w (S_n - b-c)}{v (S_n - b-c)} > C \right) > 0.$$
Indeed, this comes from the fact that $\PP(|X| \leq \delta) > 0$ for all $\delta>0$ and that the function $w (\cdot-b-c)/v (\cdot-b-c)$ is continuous; hence, the random walk can stay for an arbitrary amount of steps in a small neighborhood of $S_0=0$ subset of $Y_{>1}^{c,b}$.
It follows that we have, for a sufficiently large $N>0$,
\begin{eqnarray*}\label{ineq:contrad 2}
\PP\left( \tau_b=N \ ; \ \prod_{n=1}^{N-1} \frac{w (S_n - b-c)}{v (S_n - b-c)} > \frac{v (S_{N}-b-c)}{w (S_0-b-c)} \right) > 0,
\end{eqnarray*}
since $X$ has unbounded right support and $v (\cdot-b-c) \leq 1$.
Using then absolute continuity of $\PP$ with respect to $\QQ^\bc$ over paths with finite number of steps, we get that
\begin{eqnarray}\label{ineq:contrad 3}
\QQ^\bc \left( \tau_b=N \ ; \ \prod_{n=1}^{N-1} \frac{w (S_n - b-c)}{v (S_n - b-c)} > \frac{v (S_{N}-b-c)}{w (S_0-b-c)} \right) > 0.
\end{eqnarray}
We conclude then that~\eqref{ineq:contrad 1} also holds, since the event in~\eqref{ineq:contrad 3} is subset of the event in~\eqref{ineq:contrad 1}.
This proves inequality~\eqref{ineq:contrad 1}, which is a contradiction with hypothesis ($\text{S1}^\bc$).


\paragraph{Case 2: if $c$ and $b \geq 0$ are such that $S_0 = 0 \notin Y_{>1}^{c,b}$.}
The idea for this case is to reduce it to the previous one, by constructing paths that, first, move to the set $Y_{>1}^{c,b}$, and second, spend a sufficiently large amount of time in $Y_{>1}^{c,b}$.
For that, first define $\tau_{>1} := \inf \{ n \geq 0 \, : \, S_n \in Y_{>1}^{c,b} \}$.
Take then a compact set $A \subseteq \R$ and a large enough $M>0$, so that they satisfy
\begin{eqnarray*}
\PP \left( \tau_{>1} = M ; \ S_n - b \in A \text{ for }n=0,\ldots, M \right)>0;
\end{eqnarray*}
here, $A$ is chosen to satisfy $S_0-b \in \text{int}(A)$ and $Y_{>1}^{c,b} \cap \text{int}(A) \neq \emptyset$.
Note now that it holds that $\sup \left\{ \prod_{n=1}^{M} v (y_n - b-c) / w (y_n - b-c) \, : \, y_0, \ldots, y_M \in A+b \right\} $ is finite, since $A$ is compact and $v (\cdot-c) / w (\cdot-c)$ is continuous.
With this, and using the same arguments of Case 1, we obtain that there exists a large enough $N>0$ such that
\begin{eqnarray}\label{ineq:contrad 4}
\PP\left( \tau_b=N+M \ ; \ \frac{v (S_{0}-b-c)}{w (S_{N+M}-b-c)} \prod_{n=1}^{M+N-1} \frac{w (S_n - b-c)}{v (S_n - b-c)} > 1 \right) > 0.
\end{eqnarray}
Indeed, it is sufficient to condition the probability on the left hand side of~\eqref{ineq:contrad 4} on the event $\{ \tau_{>1} = M ; \ S_0, \ldots, S_M \in A+b \}$ and use the strong Markov property.
It follows that, by absolute continuity of $\PP$ with respect to $\QQ^\bc$ over paths with finite number of steps, we have that 
\begin{eqnarray}\label{ineq:contrad 5}
\QQ^\bc \left( \tau_b=N+M \ ; \ \frac{v (S_{0}-b-c)}{w (S_{N+M}-b-c)} \prod_{n=1}^{M+N-1} \frac{w (S_n - b-c)}{v (S_n - b-c)} > 1 \right) > 0.
\end{eqnarray}
Clearly then inequality~\eqref{ineq:contrad 1} holds, since the event of the latter inequality contains the event in~\eqref{ineq:contrad 5}.
We have arrived to a contradiction with hypothesis ($\text{S1}^\bc$).
\end{proof}

\section{Proof of Proposition~\ref{prop:infinite hitting time}}\label{app:hitting}


Before showing the proof of Proposition~\ref{prop:infinite hitting time} we establish the following lemma, which is a direct corollary of~\cite[Lemma~2]{blanchet2012efficient}.
It will be used to prove part (iii) of the latter result.

\begin{lemma}\label{lemma:linear time}
Let $\QQ$ be a measure over paths of $\{S_n\}$.
Assume that we have for any large enough $b>0$ that
\begin{eqnarray}\label{ineq:linear hypothesis}
\liminf_{y \to -\infty} \left\{ \EE^\QQ \left[ S_1-S_0 | S_0=y \right] - \int_{b-y}^\infty \QQ (S_1-S_0 > u | S_0=y) \ \dd u \right\} > 0 .
\end{eqnarray}
Then $\EE^\QQ \tau_b = O(b)$ as $b \to \infty$.
\end{lemma}

\begin{proof}
The proof consists in showing that if~\eqref{ineq:linear hypothesis} holds then the function $h(y) := (C + |b-y|) \I{y \leq b}$ satisfies the hypothesis of~\cite[Lemma~2]{blanchet2012efficient}, for some $C>0$.
That is, $\EE_y^{\QQ^\bc} \left[ h(S_1) \right] - h(y) < -\rho$ for all $y \leq b$ and for some $\rho>0$.
For that, note that for $y < b$ we have
\begin{eqnarray*}
\lefteqn{\EE^{\QQ} [h(S_1) - h(y) | S_0=y ] = } \\
&&-\EE^{\QQ} \left[ S_1-S_0 | S_0=y \right] + \int_{b-y}^\infty \QQ \left( S_1-S_0 > u | S_0=y \right) \ \dd u -C \QQ (S_1>b | S_0=y)
\end{eqnarray*}
Therefore,  inequality $\limsup_{y \to -\infty} \EE^{\QQ} [h(S_1) - h(y)| S_0=y] < 0$ is equivalent to inequality~\eqref{ineq:linear hypothesis}.
Using then~\cite[Lemma~2]{blanchet2012efficient} we conclude that $\EE^{\QQ} \tau_b \leq h(0)/\rho = C/\rho + \rho^{-1} b = O(b)$.
\end{proof}


\begin{proof}[Proof of Proposition~\ref{prop:infinite hitting time}]
For part (i), we have to show that $\QQ^\bc (\tau_b < \infty | S_0=0) = 1$ holds for all $b \geq 0$.
We will actually show that this is true for all $c \in \R$.

For that, first consider any $c \in \R$ and note that from~\cite[Lemma~1]{blanchet2008efficient} we have that $\lim_{y \to -\infty} \EE^{\QQ^\bc} \left[ S_1-S_0 | S_0 = y \right] > 0$.
This result applies in our case because $X$ is strongly subexponential, since $X$ has regularly varying right tails with tail index $\alpha>1$.
Also, it can be checked that
\begin{eqnarray}\label{eq 5}
\EE^{\QQ^\bc} \left[ S_1-S_0 \, | \, S_0 = y \right] = \EE^\PP \left[ X \, | \, X+Z>c+b-y \right]
\end{eqnarray}
holds, since
\begin{eqnarray}\label{eq 6}
\QQ^\bc \left( S_1-S_0 \in \cdot | S_0 = y \right) = \PP \left( X \in  \cdot | X+Z>c+b-y \right),
\end{eqnarray}
where $X$ and $Z$ are independent and $Z$ has the residual life distribution of $X$.

We have thus that there exists $\epsilon>0$ and $y_0 \in \R$ such that for all $y \leq y_0$ we have $\EE^{\QQ^\bc} \left[ S_1-S_0 | S_0 = y \right] > \epsilon$.
It follows that $\QQ^\bc (\tau_{y_0} < \infty | S_0=y)=1$ holds for all $y \leq y_0$.
We distinguish two cases now: if $b \leq y_0$ and if $y_0 < b$.
In the former case, it is direct that $\QQ^\bc (\tau_{b} < \infty | S_0=0) \geq \QQ^\bc (\tau_{y_0} < \infty | S_0=0) = 1$ holds, since $0 \leq b \leq y_0$.
In the latter case on the other hand, that is if $y_0 < b$, we can use a standard geometric trials argument to get that, for all \guidonote{$y \leq y_0$},
$$\QQ^\bc (\tau_b < \infty | S_0=y) \geq \QQ^\bc (\text{Geom}(\gamma) < \infty) = 1.$$
Here, $\text{Geom}(\gamma)$ is an independent geometric random variable with parameter
$$\gamma := \inf_{\overline{y} \in [y_0, b]} \QQ^\bc \left( S_1-S_0 > b-y_0 | S_0 = \overline{y} \right),$$
\guidonote{where $\gamma>0$ because of~\eqref{eq 6} and using that the function $\overline y \mapsto \PP(X>c+b-y | X+Z>c+b-\overline y)$ is continuous and strictly positive.
Indeed, it is continuous because both $X$ and $Z$ have absolutely continuous distributions, the former by Assumption~(A3) and the latter by definition~\eqref{def:Z}.}
In both cases\guidonote{, $b \leq y_0$ and $y_0 < b$,} we have shown that~$\QQ^\bc (\tau_b < \infty | S_0=y) = 1$.


For (ii), we proceed by contradiction and assume that $\EE^{\QQ^\bc} \left[ \tau_b \right] < \infty$.
By~\cite[Theorem~1.1]{asmussen1996large}, for $\alpha \in (1,2)$ it holds $\EE^\PP \left[ \tau_b \left| \tau_b < \infty \right. \right] = \infty$.
But since $\QQ^\bc$ is \guidonote{a direct proposal for exact} conditional sampling when $\alpha \in (1,3/2)$, then by Corollary~\ref{prop:efficiency} part (ii) we have
$$\EE^{\QQ^\bc} \left[ \tau_b \right] = \EE^{\QQ^\bc} \left[ \sum_{n \geq 0} n \I{\tau_b=n} \right] \geq \EE^\PP \left[ \sum_{n \geq 0} n \I{\tau_b=n} \right] = \EE^\PP \left[ \tau_b \I{\tau_b < \infty} \right] = \infty ,$$
which is a contradiction.


For (iii), by Lemma~\ref{lemma:linear time} it is sufficient to show that~\eqref{ineq:linear hypothesis} holds.
For that, note that by definition of the Blanchet-Glynn kernel in~\eqref{def:Q} and the fact that
\begin{eqnarray}\label{eq 3}
v (y) = \PP(Z>-y) \guidonote{= \frac{1}{|\EE^\PP X |} \EE^\PP \left[ [X+y]^+ \right]}
\end{eqnarray}
and
\begin{eqnarray}\label{eq 4}
w (y) = \PP(X+Z>-y) \guidonote{= \EE^\PP \left[ v(y+X) \right]}
\end{eqnarray}
we have
\begin{eqnarray*}
\QQ^\bc \left( S_1-S_0 \in \cdot \ | \ S_0 = y \right) &=& \PP \left( X \in \cdot \ | \ X+Z>-y+c \right),
\end{eqnarray*}
where $Z$ is independent of $X$ and has the residual life distribution of $X$.
With this \guidonote{and identities~\eqref{eq 5} and~\eqref{eq 6}} we get that
\begin{eqnarray}
\lefteqn{\EE^{\QQ^\bc} \left[ S_1-S_0 | S_0=y \right] - \int_{b-y}^\infty \QQ^\bc (S_1-S_0 > u | S_0=y) \ \dd u} \nonumber \\
&=& \EE^\PP \left[ X | X+Z>c \guidonote{+b}-y \right] - \int_{b-y}^\infty \EE^\PP \left[ \I{X>u} | X+Z>c \guidonote{+b}-y \right] \ \dd u \nonumber \\
&=& \EE^\PP \left[ X | X+Z>c \guidonote{+b}-y \right] - \EE^\PP \left[ \left. \left[ X-b+y \right]^+ \right| X+Z>c \guidonote{+b}-y \right]. \label{eq 2}
\end{eqnarray}
For the first term in the right-hand side of~\eqref{eq 2} one obtains that
$$\liminf_{y \to  -\infty} \ \EE^\PP \left[ X | X+Z>c\guidonote{+b}-y \right] \geq (\alpha-1) |\EE^\PP X|$$
by following the same arguments of the proof of~\cite[Lemma 1]{blanchet2008efficient} and using that $t \cdot \PP(X>t) \sim |\EE^\PP X| (\alpha-1) \PP(Z>t)$ as $t \to \infty$, which is direct from Karamata's Theorem, see~\cite[Theorem 1.6.1]{bingham1989regular}.
For the second term in the right-hand side \guidonote{of~\eqref{eq 2}} one gets that if \guidonote{$b-y>z_0$} then \guidonote{by identities~\eqref{eq 3} and~\eqref{eq 4} we have}
\begin{eqnarray*}
\lefteqn{ \EE^\PP \left. \left[ \left[ X-b+y \right]^+ \right| X+Z>c \guidonote{+b}-y \right] = \frac{\EE^\PP \left[ \left[ X-b+y \right]^+ \guidonote{\I{X+Z>c+b-y}} \right] }{ \PP(X+Z>c+b-y) } } \\
&& \guidonote{= \frac{v (y-b-c) }{ w (y-b-c) } \frac{\int_{b-y}^\infty \PP(X>u, \ X+Z>c+b-y)\ \dd u}{\PP(Z>c+b-y)} }\\
&& \guidonote{ \leq \frac{v (y-b-c) }{ w (y-b-c) } \frac{\int_{b-y}^\infty \PP(X>u)\ \dd u}{\PP(Z>c+b-y)}} = \frac{v (y-b-c) }{ w (y-b-c) } \guidonote{\frac{\PP(Z>b-y)}{\PP(Z>c+b-y)}} |\EE^\PP X|.
\end{eqnarray*}
It follows that
$$\limsup_{y \to -\infty} \EE^\PP \left[ \left[ X-b+y \right]^+ | X+Z>\guidonote{b}-y+c \right] \guidonote{\leq} |\EE^\PP X|,$$
since $v (y-b-c) / w (y-b-c) \to 1$ as $y \to -\infty$ by~\cite[Proposition 3]{blanchet2008efficient}\guidonote{, and since $Z$ is long tailed because it is regularly varying with tail index $\alpha-1>0$}.

We have thus obtained that
$$\liminf_{y \to -\infty} \left\{ \EE^{\QQ^\bc} \left[ S_1-S_0 | S_0=y \right] - \int_{b-y}^\infty \QQ^\bc (S_1-S_0 > u | S_0=y) \dd u \right\} \geq (\alpha-2) \cdot |\EE^\PP X|,$$
so applying Lemma~\ref{lemma:linear time} we conclude that if $\alpha>2$ then $\EE^{\QQ^\bc} \tau_b = O(b)$ as $b \to \infty$.
\end{proof}

\end{appendices}

\bibliographystyle{plain}
\bibliography{bibliography}

\end{document}